\documentclass[11pt]{article}
\usepackage[margin=1in]{geometry}
\usepackage{amsmath,amssymb,amsthm,mathtools}
\usepackage[T1]{fontenc}
\usepackage{lmodern}
\usepackage{microtype}
\usepackage{hyperref}

\newcommand\ZZ{\mathbb Z}

\newcommand\CC{\mathbb C}

\theoremstyle{plain}
\newtheorem{theorem}{Theorem}
\newtheorem{lemma}{Lemma}
\newtheorem{proposition}{Proposition}
\newtheorem{corollary}{Corollary}

\theoremstyle{definition}

\newcommand{\Ga}{\Gamma}
\newcommand{\Hyp}{\,{}_2F_1}

\begin{document}
\title{Proof of Sun's conjectures on hyperbolic cosine series via the Eisenstein--Lambert method}
\author{Nikita Kalinin, Guangdong Technion – Israel Institute of Technology, China}

\maketitle
%
%
%




\begin{abstract}
We prove two conjectures of Zhi-Wei Sun concerning hyperbolic cosine Lambert series. The first one is the evaluation, for integers \(m\geq 0\), of
\[
S_m=\sum_{n=1}^\infty\left(
\frac{n^{2m}}{\cosh(\pi n)-1}
-\frac{(2^{2m+1}-(-1)^{m(m+1)/2}2^{m+1}+4)n^{2m}}{\cosh(2\pi n)-1}
+\frac{2^{2m+2}n^{2m}}{\cosh(4\pi n)-1}
\right).
\]
We prove that
\[
S_0=\frac1{12},\qquad
S_1=\frac1{2\pi^2},\qquad
S_m=0\quad (m>1).
\]
The second one is the quadratic identity
\[
\sum_{n=1}^\infty \left(
\frac{4}{(\cosh(\pi n)-1)^2}
-\frac{55}{(\cosh(2\pi n)-1)^2}
+\frac{16}{(\cosh(4\pi n)-1)^2}
\right)
=
\frac{77-234/\pi}{72}.
\]
The proof uses an elementary level-four identity for Lambert series and its consequences for Eisenstein series. After differentiating this identity and evaluating it at \(i/2\), the first conjecture follows from the modular transformation law for \(E_{2m}\), with the cases \(m=0\) and \(m=1\) treated separately by the quasimodular transformation law for \(E_2\). The second conjecture follows by rewriting the corresponding squared-kernel
series as
\((E_4+10E_2-11)/360\)
and evaluating only the resulting linear combinations of \(E_2\) and \(E_4\) at \(i/2\), \(i\), and \(2i\).

2020 MSC Primary: 11F11, 11F27; Secondary: 33E05.

keywords: Sun's conjectures, hyperbolic cosine series, Lambert series, Eisenstein series, modular forms, quasimodular forms.

\end{abstract}


\section{Introduction}

The starting point of this paper is a small cluster of recent experimental identities \cite{Pov} involving the hyperbolic cosine function at the special arguments \(\pi n\), \(2\pi n\), and \(4\pi n\). The first of these is an identity  for Catalan's constant,
\[
G=\sum_{k=0}^{\infty}\frac{(-1)^k}{(2k+1)^2},
\]
namely
\[
\sum_{n=1}^{\infty}\frac1{n^2}\left(
\frac{11}{\cosh(\pi n)-1}
+\frac{11}{\cosh(4\pi n)-1}
-\frac{71}{2(\cosh(2\pi n)-1)}
\right)=G.
\]

Using the elementary identity
\[
\cosh x-1=2\sinh^2(x/2),
\]
these series can be rewritten as Lambert series, which places them naturally in the setting of Eisenstein series.

A second discussion \cite{SunMO508945} led to a companion identity of a more classical type,
\[
\sum_{n=1}^{\infty}\frac1{n^2}\left(
\frac{20}{\cosh(\pi n)-1}
-\frac{70}{\cosh(2\pi n)-1}
+\frac{20}{\cosh(4\pi n)-1}
\right)=\frac{\pi^2}{6},
\]
and suggested a broader family of identities in which analogous linear combinations are expected to produce rational multiples of \(\zeta(2m)\). In particular, the coefficients occurring in the term with \(\cosh(2\pi n)-1\) appear to follow a striking arithmetic pattern. This suggested that the phenomenon is not accidental but is governed by the level-four geometry of the CM points \(i/2\), \(i\), and \(2i\).

Motivated by the two previous observations, Zhi-Wei Sun formulated \cite{SunMO508960} a conjectural family in which nonnegative even powers \(n^{2m}\), rather than reciprocal powers, appear in the numerator.  In the same post Sun also proposed a second conjecture involving the squared denominators \((\cosh(a\pi n)-1)^2\).  Our main result proves this conjectural evaluation-and-vanishing phenomenon in full generality.

\begin{theorem}\label{thm:main}
For \(m\in\mathbb Z_{\ge 0}\), define
\[
S_m:=\sum_{n=1}^\infty\left(
\frac{n^{2m}}{\cosh(\pi n)-1}
-\frac{c_m n^{2m}}{\cosh(2\pi n)-1}
+\frac{2^{2m+2}n^{2m}}{\cosh(4\pi n)-1}
\right),
\]
where
\[
c_m=2^{2m+1}-(-1)^{m(m+1)/2}2^{m+1}+4.
\]
Then
\[
S_0=\frac1{12},\qquad
S_1=\frac1{2\pi^2},\qquad
S_m=0 \quad \text{for all } m>1.
\]
\end{theorem}

The proof is based on a level-four identity for the Lambert series
\[
\mathfrak L_k(\tau):=
\sum_{n=1}^{\infty}
\frac{n^{k-1}}{e^{-2\pi i n\tau}-1}.
\]
A direct separation of even and odd summation indices gives
\[
\mathfrak L_k(\tau)
+
\mathfrak L_k\!\left(\tau+\frac12\right)
-
(2^k+2)\mathfrak L_k(2\tau)
+
2^k\mathfrak L_k(4\tau)
=
0.
\]
For positive even weights this gives the corresponding identity for the
Eisenstein series \(E_{2j}\). After differentiating at \(\tau=i/2\), the
general case \(m>1\) follows from the modular transformation law for
\(E_{2m}\) and its consequence at the elliptic point \(i\). The two exceptional
cases \(m=0\) and \(m=1\) are treated separately using the elementary reduction
to \(E_2\) and the quasimodular transformation formula for \(E_2\).

Classical evaluations of reciprocal hyperbolic series go back at least to
Zucker \cite{Zucker1979,Zucker1984}. In the present paper the required
identities are obtained from Eisenstein--Lambert identities rather than from
individual evaluations of the Eisenstein series involved.

We shall also prove the following quadratic identity, which is the second
conjecture of Zhi-Wei Sun considered here.

\begin{theorem}\label{thm:squared}
One has
\[
\sum_{n=1}^\infty \left(
\frac{4}{(\cosh(\pi n)-1)^2}
-\frac{55}{(\cosh(2\pi n)-1)^2}
+\frac{16}{(\cosh(4\pi n)-1)^2}
\right)
=
\frac{77-234/\pi}{72}.
\]
\end{theorem}

The proof of Theorem~\ref{thm:squared} uses the elementary identity
\[
\sum_{n=1}^{\infty}
\frac{1}{\bigl(\cosh(2\pi i n\tau)-1\bigr)^2}
=
\frac{E_4(\tau)+10E_2(\tau)-11}{360}.
\]
Thus the desired series is reduced to one linear combination of \(E_4\) and
one linear combination of \(E_2\) at the three points \(i/2\), \(i\), and
\(2i\). The \(E_4\)-combination vanishes by the same level-four Eisenstein
identity, while the \(E_2\)-combination is evaluated by the quasimodular
transformation formula for \(E_2\). For related background on Ramanujan-type elliptic-function evaluations
and theta-function methods, see \cite{Berndt2016,Cooper2017}.

The paper is organized as follows. In Section~\ref{sec:def} we recall the
needed facts about Lambert series, Eisenstein series, and the transformation
formula for \(E_2\), and we prove the level-four Lambert identity. In
Section~\ref{sec:ini} we prove Theorem~\ref{thm:squared} and the exceptional
cases \(m=0,1\) of Theorem~\ref{thm:main}. In Section~\ref{sec:last} we prove
the vanishing \(S_m=0\) for all \(m>1\). In Section~\ref{sec:more} we record a general reduction for higher reciprocal kernels, a cubic companion identity, and a short alternative verification of the
low-weight cases \(m=0,1,2,3\).  The case \(m=0\) follows from
Xu and Zhao's identity~\cite[(101)]{XuZhao-2022}, while the cases
\(m=1,2,3\) follow from their explicit formulas in Corollary~2, together
with the quadratic transformation~\cite[(30)]{XuZhao-2022}.

Old (with longer proofs) version of the text is kept here as the arxiv version, see Section~\ref{sec:defo}.

\section{Definitions and standard identities}\label{sec:def}

\subsection{Eisenstein and Lambert series}

Let
\[
\mathfrak H:=\{\tau\in\mathbb C:\operatorname{Im}\tau>0\},
\qquad
q=e^{2\pi i\tau}.
\]
For \(r\geq 0\), write
\[
\sigma_r(n):=\sum_{d\mid n}d^r.
\]
We write \(B_n\) for the Bernoulli numbers, defined by
\[
\frac{t}{e^t-1}=\sum_{n=0}^{\infty}B_n\frac{t^n}{n!}.
\]

We use the Eisenstein series in Ramanujan's normalization
\cite{Rama1916}; see also Berndt's account of Ramanujan's notation
and identities~\cite[pp.~87--142]{B1991}:
\[
E_2(\tau)=1-24\sum_{n=1}^{\infty}\sigma_1(n)q^n,
\]
and, for \(j\geq 2\),
\[
E_{2j}(\tau)
=
1-\frac{4j}{B_{2j}}\sum_{n=1}^{\infty}\sigma_{2j-1}(n)q^n.
\]
The same formula also gives \(E_2\) when \(j=1\), since \(B_2=1/6\).

For \(j\geq 2\), we shall use the modular transformation laws
\[
E_{2j}(\tau+1)=E_{2j}(\tau),
\qquad
E_{2j}\!\left(-\frac1\tau\right)=\tau^{2j}E_{2j}(\tau).
\]
For \(E_2\), we use the quasimodular transformation law
\[
E_2\!\left(-\frac1\tau\right)
=
\tau^2E_2(\tau)+\frac{6\tau}{\pi i}.
\]
Differentiating the modular transformation formula gives, for every even
\(k\geq 4\),
\[
E_k'\!\left(-\frac1\tau\right)
=
k\tau^{k+1}E_k(\tau)+\tau^{k+2}E_k'(\tau),
\]
where the prime denotes differentiation with respect to \(\tau\).

We shall use the Lambert series
\[
\mathfrak L_k(\tau):=
\sum_{n=1}^{\infty}
\frac{n^{k-1}}{e^{-2\pi i n\tau}-1}
=
\sum_{n=1}^{\infty}
n^{k-1}\frac{q^n}{1-q^n}.
\]
For the values of \(k\) used below, these series converge absolutely and
locally uniformly on \(\mathfrak H\). Moreover, for \(j\geq 1\),
\[
E_{2j}(\tau)
=
1-\frac{4j}{B_{2j}}\mathfrak L_{2j}(\tau).
\]

\subsection{A level-four Lambert identity}

\begin{lemma}\label{lem:Lambert-level-four}
Let \(k\in\mathbb Z\). Then
\[
\mathfrak L_k(\tau)
+
\mathfrak L_k\!\left(\tau+\frac12\right)
-
(2^k+2)\mathfrak L_k(2\tau)
+
2^k\mathfrak L_k(4\tau)
=
0.
\]
\end{lemma}

\begin{proof}
Put \(q=e^{2\pi i\tau}\). Since \(|q|<1\), the following rearrangements are
justified by absolute convergence. We have
\[
\mathfrak L_k(\tau)
=
\sum_{n=1}^{\infty}
n^{k-1}\frac{q^n}{1-q^n},
\]
and
\[
\mathfrak L_k\!\left(\tau+\frac12\right)
=
\sum_{n=1}^{\infty}
n^{k-1}\frac{(-q)^n}{1-(-q)^n}.
\]
Therefore
\[
\mathfrak L_k(\tau)+\mathfrak L_k\!\left(\tau+\frac12\right)
=
\sum_{n=1}^{\infty}
n^{k-1}
\left(
\frac{q^n}{1-q^n}
+
\frac{(-q)^n}{1-(-q)^n}
\right).
\]
If \(n\) is even, the expression in parentheses is
\[
\frac{2q^n}{1-q^n}.
\]
If \(n\) is odd, it is
\[
\frac{q^n}{1-q^n}-\frac{q^n}{1+q^n}
=
\frac{2q^{2n}}{1-q^{2n}}.
\]
Thus
\[
\mathfrak L_k(\tau)+\mathfrak L_k\!\left(\tau+\frac12\right)
=
2\sum_{\substack{n\geq 1\\ n\ \mathrm{even}}}
n^{k-1}\frac{q^n}{1-q^n}
+
2\sum_{\substack{n\geq 1\\ n\ \mathrm{odd}}}
n^{k-1}\frac{q^{2n}}{1-q^{2n}}.
\]
The even part is
\[
2\sum_{r=1}^{\infty}
(2r)^{k-1}\frac{q^{2r}}{1-q^{2r}}
=
2^k\mathfrak L_k(2\tau).
\]
The odd part is
\[
2\sum_{n=1}^{\infty}
n^{k-1}\frac{q^{2n}}{1-q^{2n}}
-
2\sum_{r=1}^{\infty}
(2r)^{k-1}\frac{q^{4r}}{1-q^{4r}},
\]
and therefore equals
\[
2\mathfrak L_k(2\tau)-2^k\mathfrak L_k(4\tau).
\]
Consequently
\[
\mathfrak L_k(\tau)+\mathfrak L_k\!\left(\tau+\frac12\right)
=
(2^k+2)\mathfrak L_k(2\tau)-2^k\mathfrak L_k(4\tau),
\]
which is equivalent to the claimed identity.
\end{proof}

\begin{corollary}\label{cor:Eisenstein-level-four}
For every \(j\geq 1\),
\[
E_{2j}(\tau)
+
E_{2j}\!\left(\tau+\frac12\right)
-
(2^{2j}+2)E_{2j}(2\tau)
+
2^{2j}E_{2j}(4\tau)
=
0.
\]
\end{corollary}

\begin{proof}
Using
\[
E_{2j}(\tau)
=
1-\frac{4j}{B_{2j}}\mathfrak L_{2j}(\tau),
\]
the assertion follows immediately from Lemma~\ref{lem:Lambert-level-four}.
The constant terms cancel because
\[
1+1-(2^{2j}+2)+2^{2j}=0.
\]
\end{proof}

\begin{lemma}\label{lem:Lm-Lambert}
Let \(m\in\mathbb Z_{\ge 0}\). For \(a>0\), define
\[
L_m(a):=
\sum_{n=1}^{\infty}
\frac{n^{2m}}{\cosh(a\pi n)-1}.
\]
Then
\[
L_0(a)=\frac{1-E_2(ia/2)}{12},
\]
and, for \(m\geq 1\),
\[
L_m(a)=-\frac{B_{2m}}{4m\pi i}\,E_{2m}'(ia/2).
\]
\end{lemma}

\begin{proof}
Put \(\tau=ia/2\), and write
\[
Q=e^{-a\pi}=e^{2\pi i\tau}.
\]
Since
\[
\cosh(a\pi n)-1
=
\frac{Q^{-n}+Q^n-2}{2}
=
\frac{(1-Q^n)^2}{2Q^n},
\]
we have
\[
\frac{1}{\cosh(a\pi n)-1}
=
\frac{2Q^n}{(1-Q^n)^2}.
\]
For \(m=0\), this gives
\[
L_0(a)=2\sum_{n=1}^{\infty}\frac{Q^n}{(1-Q^n)^2}
=
2\sum_{r=1}^{\infty}\sigma_1(r)Q^r
=
\frac{1-E_2(\tau)}{12}.
\]

For \(m\geq 1\),
\[
L_m(a)
=
2\sum_{n=1}^{\infty}n^{2m}\frac{Q^n}{(1-Q^n)^2}.
\]
The differentiated series is also absolutely and locally uniformly convergent on \(\mathfrak H\), so termwise differentiation is justified, so
\[
\mathfrak L_{2m}'(\tau)
=
2\pi i
\sum_{n=1}^{\infty}n^{2m}\frac{Q^n}{(1-Q^n)^2}.
\]
Hence
\[
L_m(a)=\frac{1}{\pi i}\mathfrak L_{2m}'(\tau).
\]
Since
\[
E_{2m}(\tau)=1-\frac{4m}{B_{2m}}\mathfrak L_{2m}(\tau),
\]
we get
\[
\mathfrak L_{2m}'(\tau)
=
-\frac{B_{2m}}{4m}E_{2m}'(\tau).
\]
Therefore
\[
L_m(a)=-\frac{B_{2m}}{4m\pi i}E_{2m}'(ia/2).
\]
\end{proof}

\section{The quadratic identity and the exceptional cases}\label{sec:ini}

We first prove the identity needed for the reciprocal quadratic kernel.

\begin{lemma}\label{lem:squared-kernel}
For \(\tau\in\mathfrak H\), define
\[
\Phi(\tau):=
\sum_{n=1}^{\infty}
\frac{1}{\bigl(\cosh(2\pi i n\tau)-1\bigr)^2}.
\]
Then
\[
\Phi(\tau)=\frac{E_4(\tau)+10E_2(\tau)-11}{360}.
\]
\end{lemma}

\begin{proof}
Put \(q=e^{2\pi i\tau}\). Since \(\operatorname{Im}\tau>0\), we have
\(|q|<1\), and all rearrangements below are justified by absolute convergence.
Now
\[
\cosh(2\pi i n\tau)-1
=
\frac{q^n+q^{-n}-2}{2}
=
\frac{(1-q^n)^2}{2q^n}.
\]
Therefore
\[
\frac{1}{\bigl(\cosh(2\pi i n\tau)-1\bigr)^2}
=
\frac{4q^{2n}}{(1-q^n)^4}.
\]
Using
\[
\frac{1}{(1-y)^4}
=
\sum_{r=0}^{\infty}\binom{r+3}{3}y^r,
\]
we get
\[
\frac{4q^{2n}}{(1-q^n)^4}
=
4\sum_{r=0}^{\infty}\binom{r+3}{3}q^{(r+2)n}.
\]
Writing \(\ell=r+2\), this becomes
\[
\frac{4q^{2n}}{(1-q^n)^4}
=
4\sum_{\ell=2}^{\infty}\binom{\ell+1}{3}q^{\ell n}
=
\frac{2}{3}\sum_{\ell=1}^{\infty}(\ell^3-\ell)q^{\ell n}.
\]
Thus
\[
\Phi(\tau)
=
\frac{2}{3}
\sum_{n=1}^{\infty}
\sum_{\ell=1}^{\infty}
(\ell^3-\ell)q^{\ell n}.
\]
Equivalently,
\[
\Phi(\tau)
=
\frac{2}{3}
\sum_{\ell=1}^{\infty}
\frac{\ell^3-\ell}{e^{-2\pi i \ell\tau}-1}.
\]
Since
\[
E_2(\tau)
=
1-24\sum_{\ell=1}^{\infty}
\frac{\ell}{e^{-2\pi i \ell\tau}-1},
\]
and
\[
E_4(\tau)
=
1+240\sum_{\ell=1}^{\infty}
\frac{\ell^3}{e^{-2\pi i \ell\tau}-1},
\]
we obtain
\[
\Phi(\tau)
=
\frac{2}{3}
\left(
\frac{E_4(\tau)-1}{240}
-
\frac{1-E_2(\tau)}{24}
\right)
=
\frac{E_4(\tau)+10E_2(\tau)-11}{360}.
\]
\end{proof}

\begin{proof}[Proof of Theorem~\ref{thm:squared}]
The left-hand side of Theorem~\ref{thm:squared} is
\[
\mathcal S:=4\Phi(i/2)-55\Phi(i)+16\Phi(2i).
\]
By Lemma~\ref{lem:squared-kernel},
\[
360\mathcal S
=
\bigl(4E_4(i/2)-55E_4(i)+16E_4(2i)\bigr)
+
10\bigl(4E_2(i/2)-55E_2(i)+16E_2(2i)\bigr)
+385.
\]

We first evaluate the \(E_4\)-combination. By
Corollary~\ref{cor:Eisenstein-level-four} with \(j=2\),
\[
E_4(\tau)+E_4\!\left(\tau+\frac12\right)-18E_4(2\tau)+16E_4(4\tau)=0.
\]
Putting \(\tau=i/2\), we obtain
\[
E_4(i/2)+E_4\!\left(\frac{1+i}{2}\right)
-18E_4(i)+16E_4(2i)=0.
\]
Using
\[
E_4\!\left(-\frac1\tau\right)=\tau^4E_4(\tau),
\qquad
E_4(\tau+1)=E_4(\tau),
\]
we have
\[
E_4(i/2)
=
E_4\!\left(-\frac1{2i}\right)
=
(2i)^4E_4(2i)
=
16E_4(2i).
\]
Also,
\[
\frac{1+i}{2}=-\frac1{i-1},
\]
and hence
\[
E_4\!\left(\frac{1+i}{2}\right)
=
(i-1)^4E_4(i-1)
=
-4E_4(i).
\]
Substituting these two evaluations gives
\[
16E_4(2i)-4E_4(i)-18E_4(i)+16E_4(2i)=0,
\]
so
\[
16E_4(2i)-11E_4(i)=0.
\]
Therefore
\[
4E_4(i/2)-55E_4(i)+16E_4(2i)
=
80E_4(2i)-55E_4(i)
=
5\bigl(16E_4(2i)-11E_4(i)\bigr)
=
0.
\]

It remains to evaluate the \(E_2\)-combination. The transformation formula
\[
E_2\!\left(-\frac1\tau\right)
=
\tau^2E_2(\tau)+\frac{6\tau}{\pi i}
\]
gives, at \(\tau=i\),
\[
E_2(i)=\frac3\pi.
\]
At \(\tau=2i\), it gives
\[
E_2(i/2)
=
-4E_2(2i)+\frac{12}{\pi},
\]
or equivalently
\[
E_2(i/2)+4E_2(2i)=\frac{12}{\pi}.
\]
Hence
\begin{align*}
10\bigl(4E_2(i/2)-55E_2(i)+16E_2(2i)\bigr)
&=
10\left(4\bigl(E_2(i/2)+4E_2(2i)\bigr)-55E_2(i)\right)\\
&=
10\left(\frac{48}{\pi}-\frac{165}{\pi}\right)\\
&=
-\frac{1170}{\pi}.
\end{align*}
Combining the \(E_4\)- and \(E_2\)-parts, we get
\[
360\mathcal S
=
385-\frac{1170}{\pi}.
\]
Therefore
\[
\mathcal S
=
\frac{385-1170/\pi}{360}
=
\frac{77-234/\pi}{72}.
\]
This is precisely
\[
\sum_{n=1}^{\infty}
\left(
\frac{4}{(\cosh(\pi n)-1)^2}
-\frac{55}{(\cosh(2\pi n)-1)^2}
+\frac{16}{(\cosh(4\pi n)-1)^2}
\right)
=
\frac{77-234/\pi}{72}.
\]
The theorem follows.
\end{proof}

We now prove the exceptional cases \(m=0\) and \(m=1\) of
Theorem~\ref{thm:main}.

\begin{proposition}\label{prop:exceptional-cases}
For the sums \(S_m\) of Theorem~\ref{thm:main}, one has
\[
S_0=\frac1{12},
\qquad
S_1=\frac1{2\pi^2}.
\]
\end{proposition}

\begin{proof}
Recall that
\[
S_m=L_m(1)-c_mL_m(2)+2^{2m+2}L_m(4),
\]
where
\[
c_m=2^{2m+1}-(-1)^{m(m+1)/2}2^{m+1}+4.
\]

For \(m=0\), we have \(c_0=4\). By Lemma~\ref{lem:Lm-Lambert},
\[
L_0(a)=\frac{1-E_2(ia/2)}{12}.
\]
Therefore
\begin{align*}
S_0
&=
L_0(1)-4L_0(2)+4L_0(4)\\
&=
\frac{1-E_2(i/2)}{12}
-\frac{4(1-E_2(i))}{12}
+\frac{4(1-E_2(2i))}{12}\\
&=
\frac{1-E_2(i/2)+4E_2(i)-4E_2(2i)}{12}.
\end{align*}
Using
\[
E_2(i)=\frac3\pi,
\qquad
E_2(i/2)+4E_2(2i)=\frac{12}{\pi},
\]
we get
\[
S_0
=
\frac{1}{12}.
\]

For \(m=1\), we have \(c_1=16\). Again by Lemma~\ref{lem:Lm-Lambert},
\[
L_1(a)=-\frac{1}{24\pi i}E_2'(ia/2).
\]
Hence
\[
S_1
=
-\frac{1}{24\pi i}
\left(
E_2'(i/2)-16E_2'(i)+16E_2'(2i)
\right).
\]

By Corollary~\ref{cor:Eisenstein-level-four} with \(j=1\),
\[
E_2(\tau)+E_2\!\left(\tau+\frac12\right)-6E_2(2\tau)+4E_2(4\tau)=0.
\]
Differentiating and putting \(\tau=i/2\), we obtain
\[
E_2'(i/2)
+
E_2'\!\left(\frac{1+i}{2}\right)
-
12E_2'(i)
+
16E_2'(2i)
=
0.
\]
Thus
\[
E_2'(i/2)-16E_2'(i)+16E_2'(2i)
=
-E_2'\!\left(\frac{1+i}{2}\right)-4E_2'(i).
\]

It remains to compute \(E_2'((1+i)/2)\). Differentiating
\[
E_2\!\left(-\frac1\tau\right)
=
\tau^2E_2(\tau)+\frac{6\tau}{\pi i}
\]
gives
\[
E_2'\!\left(-\frac1\tau\right)
=
2\tau^3E_2(\tau)+\tau^4E_2'(\tau)+\frac{6\tau^2}{\pi i}.
\]
Putting \(\tau=i-1\), and using
\[
-\frac1{i-1}=\frac{1+i}{2},
\qquad
E_2(i-1)=E_2(i),
\qquad
E_2'(i-1)=E_2'(i),
\]
we find
\[
E_2'\!\left(\frac{1+i}{2}\right)
=
2(i-1)^3E_2(i)+(i-1)^4E_2'(i)+\frac{6(i-1)^2}{\pi i}.
\]
Since
\[
(i-1)^2=-2i,
\qquad
(i-1)^3=2+2i,
\qquad
(i-1)^4=-4,
\qquad
E_2(i)=\frac3\pi,
\]
this simplifies to
\[
E_2'\!\left(\frac{1+i}{2}\right)
=
\frac{12i}{\pi}-4E_2'(i).
\]
Consequently
\[
E_2'(i/2)-16E_2'(i)+16E_2'(2i)
=
-\frac{12i}{\pi}.
\]
Therefore
\[
S_1
=
-\frac{1}{24\pi i}\left(-\frac{12i}{\pi}\right)
=
\frac1{2\pi^2}.
\]
This proves the proposition.
\end{proof}

\section{The general vanishing}\label{sec:last}

In this section we prove the vanishing
\[
S_m=0
\qquad (m\in\mathbb Z_{\ge 2}).
\]
The only additional point needed is the following evaluation of
\(E_{2m}'\) at the point \((1+i)/2\).

\begin{lemma}\label{lem:derivative-at-half-i}
Let \(m\in\mathbb Z_{\ge 2}\), and put
\[
\varepsilon_m:=(-1)^{m(m+1)/2}.
\]
Then
\[
E_{2m}'\!\left(\frac{1+i}{2}\right)
=
\varepsilon_m2^{m+1}E_{2m}'(i).
\]
\end{lemma}

\begin{proof}
Put \(k=2m\). Since \(m>1\), we have \(k\geq 4\), and \(E_k\) is modular of
weight \(k\). Thus
\[
E_k(\tau+1)=E_k(\tau),
\qquad
E_k\!\left(-\frac1\tau\right)=\tau^kE_k(\tau).
\]
Differentiating the second identity gives
\[
E_k'\!\left(-\frac1\tau\right)
=
k\tau^{k+1}E_k(\tau)+\tau^{k+2}E_k'(\tau).
\]
Now put \(\tau=i-1\). Since
\[
-\frac1{i-1}=\frac{1+i}{2},
\]
and since \(E_k(i-1)=E_k(i)\), \(E_k'(i-1)=E_k'(i)\), we get
\[
E_k'\!\left(\frac{1+i}{2}\right)
=
k(i-1)^{k+1}E_k(i)+(i-1)^{k+2}E_k'(i).
\]

We also need the standard consequence of the same modular equation at the
elliptic point \(i\):
\[
E_k(i)
=
\begin{cases}
\dfrac{2}{ki}E_k'(i), & k\equiv 0 \pmod 4,\\[6pt]
0, & k\equiv 2 \pmod 4.
\end{cases}
\]
Indeed, if \(k\equiv 2\pmod 4\), then
\[
E_k(i)=i^kE_k(i)=-E_k(i),
\]
so \(E_k(i)=0\). If \(k\equiv 0\pmod 4\), differentiating
\[
E_k\!\left(-\frac1\tau\right)=\tau^kE_k(\tau)
\]
at \(\tau=i\) gives
\[
-E_k'(i)=ki^{k-1}E_k(i)+i^kE_k'(i),
\]
and hence
\[
E_k(i)=\frac{2}{ki}E_k'(i).
\]

We now return to \(k=2m\). If \(m\) is odd, then \(2m\equiv 2\pmod 4\), so
\(E_{2m}(i)=0\). Therefore
\[
E_{2m}'\!\left(\frac{1+i}{2}\right)
=
(i-1)^{2m+2}E_{2m}'(i)
=
(-2i)^{m+1}E_{2m}'(i).
\]
Writing \(m=2r+1\), we have
\[
(-2i)^{m+1}
=
2^{m+1}(-i)^{2r+2}
=
2^{m+1}(-1)^{r+1},
\]
while
\[
(-1)^{m(m+1)/2}
=
(-1)^{(2r+1)(r+1)}
=
(-1)^{r+1}.
\]
Hence
\[
(-2i)^{m+1}=\varepsilon_m2^{m+1}.
\]

If \(m\) is even, then \(2m\equiv 0\pmod 4\), and
\[
E_{2m}(i)=\frac{1}{mi}E_{2m}'(i).
\]
Thus
\begin{align*}
E_{2m}'\!\left(\frac{1+i}{2}\right)
&=
2m(i-1)^{2m+1}E_{2m}(i)
+
(i-1)^{2m+2}E_{2m}'(i)\\
&=
\left(
\frac{2(i-1)^{2m+1}}{i}
+
(i-1)^{2m+2}
\right)E_{2m}'(i)\\
&=
(i-1)^{2m+1}
\left(
\frac{2}{i}+i-1
\right)E_{2m}'(i)\\
&=
-(1+i)(i-1)^{2m+1}E_{2m}'(i).
\end{align*}
Since
\[
-(1+i)(i-1)=2,
\qquad
(i-1)^{2m}=(-2i)^m,
\]
we get
\[
E_{2m}'\!\left(\frac{1+i}{2}\right)
=
2(-2i)^mE_{2m}'(i).
\]
Writing \(m=2r\), we have
\[
2(-2i)^m
=
2^{m+1}(-i)^{2r}
=
2^{m+1}(-1)^r,
\]
while
\[
(-1)^{m(m+1)/2}
=
(-1)^{r(2r+1)}
=
(-1)^r.
\]
Hence
\[
2(-2i)^m=\varepsilon_m2^{m+1}.
\]
The lemma follows.
\end{proof}

\begin{proposition}\label{prop:general-vanishing}
For every \(m\in\mathbb Z_{\ge 2}\), one has
\[
S_m=0.
\]
\end{proposition}

\begin{proof}
Let \(m>1\), and put
\[
\varepsilon_m:=(-1)^{m(m+1)/2}.
\]
Recall that
\[
c_m=2^{2m+1}-\varepsilon_m2^{m+1}+4.
\]
By definition,
\[
S_m=L_m(1)-c_mL_m(2)+2^{2m+2}L_m(4).
\]
By Lemma~\ref{lem:Lm-Lambert}, since \(m\geq 2\),
\[
L_m(a)
=
-\frac{B_{2m}}{4m\pi i}E_{2m}'(ia/2).
\]
Since \(B_{2m}\ne0\), it is enough to prove
\[
E_{2m}'(i/2)-c_mE_{2m}'(i)+2^{2m+2}E_{2m}'(2i)=0.
\]

By Corollary~\ref{cor:Eisenstein-level-four},
\[
E_{2m}(\tau)
+
E_{2m}\!\left(\tau+\frac12\right)
-
(2^{2m}+2)E_{2m}(2\tau)
+
2^{2m}E_{2m}(4\tau)
=
0.
\]
Differentiating and putting \(\tau=i/2\), we obtain
\[
E_{2m}'(i/2)
+
E_{2m}'\!\left(\frac{1+i}{2}\right)
-
(2^{2m+1}+4)E_{2m}'(i)
+
2^{2m+2}E_{2m}'(2i)
=
0.
\]
By Lemma~\ref{lem:derivative-at-half-i},
\[
E_{2m}'\!\left(\frac{1+i}{2}\right)
=
\varepsilon_m2^{m+1}E_{2m}'(i).
\]
Substituting this gives
\[
E_{2m}'(i/2)
+
\varepsilon_m2^{m+1}E_{2m}'(i)
-
(2^{2m+1}+4)E_{2m}'(i)
+
2^{2m+2}E_{2m}'(2i)
=
0.
\]
Equivalently,
\[
E_{2m}'(i/2)
-
\left(2^{2m+1}-\varepsilon_m2^{m+1}+4\right)E_{2m}'(i)
+
2^{2m+2}E_{2m}'(2i)
=
0.
\]
Since
\[
c_m=2^{2m+1}-\varepsilon_m2^{m+1}+4,
\]
this is exactly
\[
E_{2m}'(i/2)-c_mE_{2m}'(i)+2^{2m+2}E_{2m}'(2i)=0.
\]
Hence \(S_m=0\).
\end{proof}

\begin{proof}[Proof of Theorem~\ref{thm:main}]
The exceptional cases
\[
S_0=\frac1{12},
\qquad
S_1=\frac1{2\pi^2}
\]
were proved in Proposition~\ref{prop:exceptional-cases}. The remaining cases
\(m\in\mathbb Z_{\ge 2}\) are exactly Proposition~\ref{prop:general-vanishing}. This proves
Theorem~\ref{thm:main}.
\end{proof}

\section{Further identities and alternative low-weight checks}\label{sec:more}
\begin{proposition}[Higher powers of the reciprocal kernel]
The same Lambert-series method applies to all powers of the reciprocal
hyperbolic kernel. More precisely, for an integer \(r\geq 1\), define
\[
\Phi_r(\tau):=
\sum_{n=1}^{\infty}
\frac{1}{\bigl(\cosh(2\pi i n\tau)-1\bigr)^r},
\qquad \operatorname{Im}\tau>0.
\]
Let
\[
P_r(X):=
\frac{2^r}{(2r-1)!}
X\prod_{\ell=1}^{r-1}(X^2-\ell^2).
\]
Writing
\[
P_r(X)=\sum_{j=1}^{r}a_{r,j}X^{2j-1},
\]
one has the identity
\[
\Phi_r(\tau)
=
\sum_{j=1}^{r}a_{r,j}\mathfrak L_{2j}(\tau).
\]
Equivalently,
\[
\Phi_r(\tau)
=
-\sum_{j=1}^{r}
\frac{a_{r,j}B_{2j}}{4j}
\bigl(E_{2j}(\tau)-1\bigr).
\]
Thus each higher-power kernel lies in the linear span of
\(1,E_2(\tau),E_4(\tau),\ldots,E_{2r}(\tau)\). The case \(r=2\) gives
\[
P_2(X)=\frac{2}{3}(X^3-X),
\]
and hence
\[
\Phi_2(\tau)
=
\frac{2}{3}\bigl(\mathfrak L_4(\tau)-\mathfrak L_2(\tau)\bigr)
=
\frac{E_4(\tau)+10E_2(\tau)-11}{360},
\]
which is the identity used in the proof of Theorem~\ref{thm:squared}.
\end{proposition}

\begin{proof}
Put \(q=e^{2\pi i\tau}\). Since
\[
\cosh(2\pi i n\tau)-1
=
\frac{(1-q^n)^2}{2q^n},
\]
we have
\[
\frac{1}{\bigl(\cosh(2\pi i n\tau)-1\bigr)^r}
=
\frac{2^r q^{rn}}{(1-q^n)^{2r}}.
\]
Using
\[
\frac{1}{(1-y)^{2r}}
=
\sum_{s=0}^{\infty}
\binom{s+2r-1}{2r-1}y^s,
\]
we obtain
\[
\frac{2^r q^{rn}}{(1-q^n)^{2r}}
=
2^r
\sum_{s=0}^{\infty}
\binom{s+2r-1}{2r-1}q^{(s+r)n}.
\]
Writing \(m=s+r\), this becomes
\[
2^r
\sum_{m=r}^{\infty}
\binom{m+r-1}{2r-1}q^{mn}.
\]
As a polynomial in \(m\),
\[
2^r\binom{m+r-1}{2r-1}
=
\frac{2^r}{(2r-1)!}
m\prod_{\ell=1}^{r-1}(m^2-\ell^2)
=
P_r(m).
\]
Therefore (since \(P_r(m)=0, 1\leq m  <r\))
\[
\frac{1}{\bigl(\cosh(2\pi i n\tau)-1\bigr)^r}
=
\sum_{m=1}^{\infty}P_r(m)q^{mn}.
\]
Since the series are absolutely convergent for \(\operatorname{Im}\tau>0\),
we may sum over \(n\) and interchange the summations:
\[
\Phi_r(\tau)
=
\sum_{n=1}^{\infty}\sum_{m=1}^{\infty}P_r(m)q^{mn}.
\]
Using
\[
P_r(m)=\sum_{j=1}^{r}a_{r,j}m^{2j-1},
\]
we get
\[
\Phi_r(\tau)
=
\sum_{j=1}^{r}a_{r,j}
\sum_{m=1}^{\infty}
m^{2j-1}\frac{q^m}{1-q^m}
=
\sum_{j=1}^{r}a_{r,j}\mathfrak L_{2j}(\tau).
\]
Finally,
\[
E_{2j}(\tau)
=
1-\frac{4j}{B_{2j}}\mathfrak L_{2j}(\tau),
\]
so
\[
\mathfrak L_{2j}(\tau)
=
-\frac{B_{2j}}{4j}\bigl(E_{2j}(\tau)-1\bigr).
\]
This gives the Eisenstein-series form.
\end{proof}

For example, for \(r=3\),
\[
P_3(X)=\frac1{15}(X^5-5X^3+4X),
\]
and therefore
\[
\Phi_3(\tau)
=
\frac1{15}\bigl(\mathfrak L_6(\tau)-5\mathfrak L_4(\tau)+4\mathfrak L_2(\tau)\bigr)
=
\frac{191}{15120}
-\frac{E_6(\tau)}{7560}
-\frac{E_4(\tau)}{720}
-\frac{E_2(\tau)}{90}.
\]

\begin{proposition}[A cubic companion identity]
One has
\[
\sum_{n=1}^{\infty}
\left(
\frac{1}{(\cosh(\pi n)-1)^3}
-\frac{55}{(\cosh(2\pi n)-1)^3}
+\frac{64}{(\cosh(4\pi n)-1)^3}
\right)
=
\frac{191}{1512}+\frac{7}{10\pi}
-\frac{\Gamma(1/4)^4}{16\pi^3}.
\]
\end{proposition}

\begin{proof}
For \(r=3\), the polynomial in the higher-power reduction is
\[
P_3(X)=\frac1{15}(X^5-5X^3+4X).
\]
Hence
\[
\Phi_3(\tau)
=
\frac1{15}
\bigl(\mathfrak L_6(\tau)-5\mathfrak L_4(\tau)+4\mathfrak L_2(\tau)\bigr).
\]
Using
\[
E_2(\tau)=1-24\mathfrak L_2(\tau),
\qquad
E_4(\tau)=1+240\mathfrak L_4(\tau),
\qquad
E_6(\tau)=1-504\mathfrak L_6(\tau),
\]
we obtain
\[
\Phi_3(\tau)
=
\frac{191}{15120}
-\frac{E_6(\tau)}{7560}
-\frac{E_4(\tau)}{720}
-\frac{E_2(\tau)}{90}.
\]

Consider
\[
\mathcal T_3:=\Phi_3(i/2)-55\Phi_3(i)+64\Phi_3(2i).
\]
The \(E_6\)-part vanishes. Indeed,
\[
E_6(i)=0,
\qquad
E_6(i/2)=-64E_6(2i),
\]
and therefore
\[
E_6(i/2)-55E_6(i)+64E_6(2i)=0.
\]
The \(E_4\)-part also vanishes. From the level-four Eisenstein relation and
the modular transformation of \(E_4\), one has
\[
E_4(i/2)=16E_4(2i),
\qquad
16E_4(2i)=11E_4(i).
\]
Thus
\[
E_4(i/2)-55E_4(i)+64E_4(2i)
=
80E_4(2i)-55E_4(i)
=
0.
\]
Consequently
\[
\mathcal T_3
=
\frac{191}{15120}(1-55+64)
-\frac1{90}
\bigl(E_2(i/2)-55E_2(i)+64E_2(2i)\bigr).
\]
Now
\[
E_2(i)=\frac3\pi,
\qquad
E_2(i/2)+4E_2(2i)=\frac{12}{\pi},
\]
and the parameter \(R=4\) case of the Borwein--Borwein alpha-function evaluation
for \(E_2(i\sqrt R)\) gives~\cite[pp.~152, 164, Ex.~15]{BorweinBorwein1998}
\[
E_2(2i)
=
\frac{3}{2\pi}
+
\frac{3\Gamma(1/4)^4}{32\pi^3}.
\]
Therefore
\[
E_2(i/2)-55E_2(i)+64E_2(2i)
=
-\frac{63}{\pi}
+
\frac{45\Gamma(1/4)^4}{8\pi^3}.
\]
Substituting this into the previous formula gives
\[
\mathcal T_3
=
\frac{191}{1512}
+
\frac{7}{10\pi}
-
\frac{\Gamma(1/4)^4}{16\pi^3}.
\]
This is exactly the stated identity.
\end{proof}

\subsection{Relation with identities of Xu and Zhao}\label{sec:alt}

We record a short independent verification of the first few cases of
Theorem~\ref{thm:main} from the formulas of Xu and Zhao~\cite{XuZhao-2022}.
This verification is not used in the proof of the general vanishing theorem.

We use the notation
\[
S_{p,2}(y)=\sum_{n=1}^{\infty}\frac{n^p}{\sinh^2(ny)} .
\]
Since
\[
\cosh t-1=2\sinh^2(t/2),
\]
the summands in Theorem~\ref{thm:main} may be rewritten as
\[
\sum_{n=1}^{\infty}
\frac{n^{2m}}{\cosh(a\pi n)-1}
=
\frac12
\sum_{n=1}^{\infty}
\frac{n^{2m}}{\sinh^2(a\pi n/2)}
=
\frac12\,S_{2m,2}(a\pi/2).
\]

For \(m=0\), the required relation follows from Xu and Zhao's identity
\[
\alpha S_{0,2}(\alpha)+\beta S_{0,2}(\beta)
-\frac{\alpha+\beta}{6}+1=0,
\qquad \alpha\beta=\pi^2,
\]
together with the value
\[
S_{0,2}(\pi)=\frac16-\frac1{2\pi},
\]
which follows immediately from \(S_{0,2}(\pi)=2L_0(2)\) and \(E_2(i)=3/\pi\).
Taking \(\alpha=\pi/2\) and \(\beta=2\pi\), one obtains
\[
S_{0,2}(\pi/2)+4S_{0,2}(2\pi)
=
\frac56-\frac2\pi .
\]
Therefore
\[
\begin{aligned}
S_0
&=\frac12\left(
S_{0,2}(\pi/2)-4S_{0,2}(\pi)+4S_{0,2}(2\pi)
\right)  \\
&=\frac12\left(
\frac56-\frac2\pi
-4\left(\frac16-\frac1{2\pi}\right)
\right)
=\frac1{12}.
\end{aligned}
\]

For \(m=1,2,3\), Corollary~2 of Xu and Zhao gives explicit formulas for
\(S_{2m,2}(y/2)\) and \(S_{2m,2}(y)\) in terms of their variables
\(x,z,z'\).  To obtain the corresponding \(S_{2m,2}(2y)\)-formula, one uses
their quadratic transformation~\cite[(30)]{XuZhao-2022}; in particular, its
first two coordinates are
\[
x\longmapsto
\left(\frac{1-\sqrt{1-x}}{1+\sqrt{1-x}}\right)^2,
\qquad
y\longmapsto 2y.
\]
At the self-dual value \(x=1/2\) one has \(y=\pi\).  Substituting \(x=1/2\)
in the formulas of Corollary~2 and applying transformation~\cite[(30)]{XuZhao-2022}
for the \(2y\)-term gives, after direct substitution and simplification,
\[
S_1=\frac1{2\pi^2},
\qquad
S_2=0,
\qquad
S_3=0.
\]
Thus the first four cases \(m=0,1,2,3\) of Theorem~\ref{thm:main} may also
be checked from the Ramanujan-type formulas of Xu and Zhao.

\section*{Acknowledgments}
The initial inspiration for a solution of these conjectures came from AI (Chat GPT-5.4 Plus). I would like to thank Ce Xu for explaining how the low-weight cases in
Section~\ref{sec:alt} can be recovered from the formulas of Xu and Zhao and for suggesting relevant literature.
I thank an anonymous referee for suggesting succinct proofs of both theorems. I thank another anonymous referee for an advice to consider higher powers of the reciprocal kernel.

Also I decided to keep the old and more detailed version below, for the arxiv version of the article.

\section{Definitions and standard identities}\label{sec:defo}

\subsection{Eisenstein series}

Let \(\tau\in\mathfrak H:=\{\tau\in\CC:\Im\tau>0\}\) and set
\[
q=e^{2\pi i\tau}.
\]
For \(r\ge 0\), write
\[
\sigma_r(n):=\sum_{d\mid n} d^r.
\]
We write \(B_n\) for the Bernoulli numbers, defined by
\[
\frac{t}{e^t-1}=\sum_{n=0}^\infty B_n \frac{t^n}{n!},\ B_2=1/6, B_4=-1/30, B_6=1/42.
\]

We use the quasimodular Eisenstein series
\[
E_2(\tau)=1-24\sum_{n=1}^\infty \sigma_1(n)q^n,
\]
and, for every even integer \(k\ge 4\), the normalized Eisenstein series
\[
E_k(\tau)=1-\frac{2k}{B_k}\sum_{n=1}^\infty \sigma_{k-1}(n)q^n.
\]
In particular,
\begin{align*}
E_4(\tau)&=1+240\sum_{n=1}^\infty \sigma_3(n)q^n,\\
E_6(\tau)&=1-504\sum_{n=1}^\infty \sigma_5(n)q^n.
\end{align*}

We also use the Ramanujan differential operator
\[
D:=q\frac{d}{dq}=\frac{1}{2\pi i}\frac{d}{d\tau}.
\]
Ramanujan's differential identities are
\begin{align}
DE_2&=\frac{E_2^2-E_4}{12},\label{eq:RamE2}\\
DE_4&=\frac{E_2E_4-E_6}{3},\label{eq:RamE4}\\
DE_6&=\frac{E_2E_6-E_4^2}{2}.\label{eq:RamE6}
\end{align}

For every even integer \(k\ge 4\), define the classical Eisenstein series
\[
G_k(\tau):=2\zeta(k)E_k(\tau)
=\sum_{(u,v)\in\ZZ^2\setminus\{(0,0)\}} (u+v\tau)^{-k}.
\]
For \(k\ge 4\), the defining series for \(G_k\) and its termwise derivative converge absolutely and locally uniformly on \(\mathfrak H\). Moreover,
\begin{equation}
G_k\!\left(-\frac1\tau\right)=\tau^k G_k(\tau).
\label{eq:Gkmod}
\end{equation}
In particular,
\begin{align}
E_4\!\left(-\frac1\tau\right)&=\tau^4E_4(\tau),\label{eq:E4mod}\\
E_6\!\left(-\frac1\tau\right)&=\tau^6E_6(\tau).\label{eq:E6mod}
\end{align}
For \(E_2\), one has the quasimodular transformation law
\begin{equation}
E_2\!\left(-\frac1\tau\right)=\tau^2E_2(\tau)+\frac{6\tau}{\pi i}.
\label{eq:E2mod}
\end{equation}

\subsection{Jacobi theta functions}

Let
\[
Q=e^{\pi i\tau}.
\]
For background on the theta-function identities used below, see, for example, \cite[Chapter~4]{Cooper2017}. We use the classical theta functions
\begin{align*}
\vartheta_2(\tau)&:=\sum_{n\in\ZZ} Q^{(n+1/2)^2},\\
\vartheta_3(\tau)&:=\sum_{n\in\ZZ} Q^{n^2},\\
\vartheta_4(\tau)&:=\sum_{n\in\ZZ} (-1)^n Q^{n^2}.
\end{align*}
The identities we need are:
\begin{align}
\vartheta_3(\tau)^4&=\vartheta_2(\tau)^4+\vartheta_4(\tau)^4,\label{eq:Jacobi}\\
\vartheta_2(\tau)^2&=2\vartheta_2(2\tau)\vartheta_3(2\tau),\label{eq:dup1}\\
\vartheta_3(\tau)^2&=\vartheta_3(2\tau)^2+\vartheta_2(2\tau)^2,\label{eq:dup2}\\
\vartheta_4(\tau)^2&=\vartheta_3(2\tau)^2-\vartheta_2(2\tau)^2.\label{eq:dup3}
\end{align}
Moreover, under the inversion $\tau\mapsto -1/\tau$ one has
\[
\vartheta_2\!\left(-\frac1\tau\right)=(-i\tau)^{1/2}\vartheta_4(\tau),
\qquad
\vartheta_4\!\left(-\frac1\tau\right)=(-i\tau)^{1/2}\vartheta_2(\tau).
\]
At the fixed point $\tau=i$ this implies
\begin{equation}
\vartheta_2(i)=\vartheta_4(i).\label{eq:t2=t4}
\end{equation}

We also need the standard formulas expressing Eisenstein series in theta constants:
\begin{align}
2E_2(2\tau)-E_2(\tau)&=\vartheta_3(2\tau)^4+\vartheta_2(2\tau)^4,\label{eq:E2theta}\\
E_4(\tau)&=\frac12\bigl(\vartheta_2(\tau)^8+\vartheta_3(\tau)^8+\vartheta_4(\tau)^8\bigr),\label{eq:E4theta}\\
E_6(\tau)&=\frac12\bigl(\vartheta_2(\tau)^4+\vartheta_3(\tau)^4\bigr)
\bigl(\vartheta_3(\tau)^4+\vartheta_4(\tau)^4\bigr)
\bigl(\vartheta_4(\tau)^4-\vartheta_2(\tau)^4\bigr).
\label{eq:E6theta}
\end{align}

\subsection{Reduction of the hyperbolic sums to Eisenstein series}

Define, for $a>0$ and $m\ge 0$,
\[
L_m(a):=\sum_{n=1}^\infty \frac{n^{2m}}{\cosh(a\pi n)-1}.
\]
Then
\[
S_m=L_m(1)-c_mL_m(2)+2^{2m+2}L_m(4).
\]

\begin{lemma}\label{lem:Lm}
For every \(x>0\),
\[
\frac{1}{\cosh x-1}=\frac{2e^{-x}}{(1-e^{-x})^2}=2\sum_{r=1}^\infty r e^{-rx}.
\]
Consequently, if \(\tau=ia/2\) (so that \(q=e^{-a\pi}\)), then
\[
L_m(a)=2\sum_{N=1}^\infty N\sigma_{2m-1}(N)q^N.
\]
For \(m=0\),
\begin{equation}
L_0(a)=\frac{1-E_2(\tau)}{12}.
\label{eq:L0}
\end{equation}

For every \(m\ge 1\),
\begin{equation}
L_m(a)=-\frac{B_{2m}}{2m}\,D E_{2m}(\tau).
\label{eq:Lm-general}
\end{equation}
In particular,
\begin{align}
L_1(a)&=\frac{E_4(\tau)-E_2(\tau)^2}{144},\label{eq:L1}\\
L_2(a)&=\frac{E_2(\tau)E_4(\tau)-E_6(\tau)}{360},\label{eq:L2}\\
L_3(a)&=\frac{E_4(\tau)^2-E_2(\tau)E_6(\tau)}{504}.\label{eq:L3}
\end{align}
\end{lemma}

\begin{proof}
The expansion of \(1/(\cosh x-1)\) is immediate:
\[
\cosh x-1=\frac{e^x+e^{-x}-2}{2}=\frac{(1-e^{-x})^2}{2e^{-x}},
\]
so
\[
\frac{1}{\cosh x-1}=\frac{2e^{-x}}{(1-e^{-x})^2}=2\sum_{r=1}^\infty r e^{-rx}.
\]
Therefore
\[
L_m(a)=2\sum_{n=1}^\infty\sum_{r=1}^\infty r\,n^{2m} q^{nr}.
\]
Collecting the coefficient of \(q^N\) gives
\[
\sum_{nr=N} r n^{2m}
=
\sum_{n\mid N}\frac{N}{n}n^{2m}
=
N\sum_{n\mid N}n^{2m-1}
=
N\sigma_{2m-1}(N),
\]
which proves
\[
L_m(a)=2\sum_{N=1}^\infty N\sigma_{2m-1}(N)q^N.
\]

For \(m=0\), this gives
\[
L_0(a)=2\sum_{N=1}^\infty \sigma_1(N)q^N=\frac{1-E_2(\tau)}{12}.
\]

Now let \(m\ge 1\). From the definition of \(E_{2m}\),
\[
E_{2m}(\tau)=1-\frac{4m}{B_{2m}}\sum_{N=1}^\infty \sigma_{2m-1}(N)q^N,
\]
hence
\[
D E_{2m}(\tau)
=
-\frac{4m}{B_{2m}}
\sum_{N=1}^\infty N\sigma_{2m-1}(N)q^N
=
-\frac{2m}{B_{2m}}L_m(a).
\]
This proves \eqref{eq:Lm-general}.

For \(m=1\), \eqref{eq:Lm-general} gives
\[
L_1(a)=-\frac{1}{12}DE_2(\tau),
\]
and \eqref{eq:RamE2} yields \eqref{eq:L1}. For \(m=2\), one has
\[
L_2(a)=\frac{1}{120}DE_4(\tau),
\]
and \eqref{eq:RamE4} yields \eqref{eq:L2}. For \(m=3\), one has
\[
L_3(a)=-\frac{1}{252}DE_6(\tau),
\]
and \eqref{eq:RamE6} yields \eqref{eq:L3}.
\end{proof}

\subsection{Special identities at $\tau=i$ and $\tau=2i$}

Set
\[
a:=\vartheta_3(2i),\qquad b:=\vartheta_2(2i),\qquad A:=a^4=\vartheta_3(2i)^4,\qquad B:=b^4=\vartheta_2(2i)^4.
\]

\begin{lemma}\label{lem:ABrelation}
The numbers $A$ and $B$ satisfy
\begin{equation}
A^2-34AB+B^2=0.
\label{eq:ABrel}
\end{equation}
\end{lemma}

\begin{proof}
By \eqref{eq:dup1} and \eqref{eq:dup3}, evaluated at $\tau=i$,
\[
\vartheta_2(i)^2=2ab,
\qquad
\vartheta_4(i)^2=a^2-b^2.
\]
Using \eqref{eq:t2=t4}, we get
\[
2ab=a^2-b^2.
\]
Since $a,b>0$, dividing by $b^2$ yields
\[
\left(\frac{a}{b}\right)^2-2\frac{a}{b}-1=0.
\]
Thus
\[
\frac{a}{b}=1+\sqrt2.
\]
Raising to the fourth power gives
\[
\frac{A}{B}=\left(1+\sqrt2\right)^4=17+12\sqrt2.
\]
Since $(17+12\sqrt2)+(17-12\sqrt2)=34$ and $(17+12\sqrt2)(17-12\sqrt2)=1$, the ratio $A/B$ satisfies
\[
\left(\frac{A}{B}\right)^2-34\left(\frac{A}{B}\right)+1=0.
\]
Multiplying by $B^2$ gives \eqref{eq:ABrel}.
\end{proof}

\begin{lemma}\label{lem:specialvalues}
With the notation above, one has
\begin{align}
E_2(i)&=\frac{3}{\pi},\label{eq:E2i}\\
E_2(2i)&=\frac{A+B}{2}+\frac{3}{2\pi},\label{eq:E22i}\\
E_2(i/2)&=-2(A+B)+\frac{6}{\pi},\label{eq:E2half}\\
E_4(2i)&=A^2-AB+B^2,\label{eq:E42i}\\
E_4(i)&=48AB,\label{eq:E4i}\\
E_4(i/2)&=16E_4(2i),\label{eq:E4half}\\
E_6(i)&=0,\label{eq:E6i}\\
E_6(2i)&=\frac12(A+B)(2A-B)(A-2B),\label{eq:E62i}\\
E_6(i/2)&=-64E_6(2i).\label{eq:E6half}
\end{align}
\end{lemma}

\begin{proof}
Applying \eqref{eq:E2mod} with $\tau=i$ gives
\[
E_2(i)=i^2E_2(i)+\frac{6i}{\pi i}=-E_2(i)+\frac{6}{\pi},
\]
so \eqref{eq:E2i} follows.

Next, \eqref{eq:E2theta} with $\tau=i$ gives
\[
2E_2(2i)-E_2(i)=\vartheta_3(2i)^4+\vartheta_2(2i)^4=A+B.
\]
Using \eqref{eq:E2i} we obtain \eqref{eq:E22i}. Then \eqref{eq:E2mod} with $\tau=2i$ gives
\[
E_2(i/2)=E_2\!\left(-\frac{1}{2i}\right)=(2i)^2E_2(2i)+\frac{6(2i)}{\pi i}=-4E_2(2i)+\frac{12}{\pi},
\]
which, combined with \eqref{eq:E22i}, yields \eqref{eq:E2half}.

For \eqref{eq:E42i}, use \eqref{eq:E4theta} at $\tau=2i$ and Jacobi's identity \eqref{eq:Jacobi}:
\[
\vartheta_4(2i)^4=\vartheta_3(2i)^4-\vartheta_2(2i)^4=A-B.
\]
Hence
\[
E_4(2i)=\frac12\bigl(B^2+A^2+(A-B)^2\bigr)=A^2-AB+B^2.
\]

To compute $E_4(i)$, note from \eqref{eq:t2=t4} and \eqref{eq:Jacobi} that
\[
\vartheta_3(i)^4=\vartheta_2(i)^4+\vartheta_4(i)^4=2\vartheta_2(i)^4.
\]
Therefore, by \eqref{eq:E4theta},
\[
E_4(i)=\frac12\bigl(\vartheta_2(i)^8+\vartheta_3(i)^8+\vartheta_4(i)^8\bigr)
=\frac12\bigl(x^2+(2x)^2+x^2\bigr)=3x^2,
\]
where $x:=\vartheta_2(i)^4$. From \eqref{eq:dup1},
\[
x=\vartheta_2(i)^4=(2ab)^2=4a^2b^2,
\]
so
\[
x^2=16a^4b^4=16AB.
\]
Thus \eqref{eq:E4i} follows:
\[
E_4(i)=3\cdot 16AB=48AB.
\]

Equation \eqref{eq:E4half} is immediate from \eqref{eq:E4mod} with $\tau=2i$, because $i/2=-1/(2i)$:
\[
E_4(i/2)=(2i)^4E_4(2i)=16E_4(2i).
\]

For \eqref{eq:E6i}, apply \eqref{eq:E6mod} with $\tau=i$:
\[
E_6(i)=i^6E_6(i)=-E_6(i),
\]
so $E_6(i)=0$.

For \eqref{eq:E62i}, use \eqref{eq:E6theta} at $\tau=2i$. Since
\[
\vartheta_2(2i)^4=B,\qquad \vartheta_3(2i)^4=A,\qquad \vartheta_4(2i)^4=A-B,
\]
we get
\[
E_6(2i)=\frac12(A+B)(A+A-B)(A-B-B)=\frac12(A+B)(2A-B)(A-2B).
\]
Finally, \eqref{eq:E6mod} with $\tau=2i$ gives
\[
E_6(i/2)=E_6\!\left(-\frac1{2i}\right)=(2i)^6E_6(2i)=-64E_6(2i),
\]
which is \eqref{eq:E6half}.
\end{proof}

\section{Initial cases, the quadratic identity, and an alternative proof}\label{sec:ini}

Recall
\[
S_m:=\sum_{n=1}^\infty\left(\frac{n^{2m}}{\cosh(n\pi)-1}-\frac{c_m n^{2m}}{\cosh(2n\pi)-1}+\frac{2^{2m+2}n^{2m}}{\cosh(4n\pi)-1}\right),
\]
where
\[
c_m=2^{2m+1}-(-1)^{m(m+1)/2}2^{m+1}+4.
\]
Thus
\[
c_0=4,\qquad c_1=16,\qquad c_2=44,\qquad c_3=116.
\]
The first four explicit instances are
\begin{align*}
S_0&=\sum_{n=1}^\infty\left(\frac1{\cosh(n\pi)-1}-\frac4{\cosh(2n\pi)-1}+\frac4{\cosh(4n\pi)-1}\right),\\
S_1&=\sum_{n=1}^\infty\left(\frac{n^2}{\cosh(n\pi)-1}-\frac{16n^2}{\cosh(2n\pi)-1}+\frac{16n^2}{\cosh(4n\pi)-1}\right),\\
S_2&=\sum_{n=1}^\infty\left(\frac{n^4}{\cosh(n\pi)-1}-\frac{44n^4}{\cosh(2n\pi)-1}+\frac{64n^4}{\cosh(4n\pi)-1}\right),\\
S_3&=\sum_{n=1}^\infty\left(\frac{n^6}{\cosh(n\pi)-1}-\frac{116n^6}{\cosh(2n\pi)-1}+\frac{256n^6}{\cosh(4n\pi)-1}\right).
\end{align*}

We begin by proving the cases \(m=0,1,2,3\), which already exhibit the modular mechanism underlying the general theorem.

\begin{theorem}
The sums $S_0,S_1,S_2,S_3$ satisfy
\[
S_0=\frac1{12},\qquad S_1=\frac1{2\pi^2},\qquad S_2=0,\qquad S_3=0.
\]
\end{theorem}

\begin{proof}
We treat the four cases separately.

\medskip
\noindent\textbf{Case $m=0$.}
By \eqref{eq:L0},
\[
S_0=L_0(1)-4L_0(2)+4L_0(4)
=\frac{1-E_2(i/2)+4E_2(i)-4E_2(2i)}{12}.
\]
Substituting \eqref{eq:E2i} and \eqref{eq:E2half}, we get
\[
\begin{aligned}
12S_0
&=1-\left(-4E_2(2i)+\frac{12}{\pi}\right)+4\cdot \frac{3}{\pi}-4E_2(2i)\\
&=1.
\end{aligned}
\]
Hence
\[
S_0=\frac1{12}.
\]

\medskip
\noindent\textbf{Case $m=1$.}
By \eqref{eq:L1},
\[
144S_1=\bigl(E_4-E_2^2\bigr)(i/2)-16\bigl(E_4-E_2^2\bigr)(i)+16\bigl(E_4-E_2^2\bigr)(2i).
\]
Equivalently,
\[
144S_1=E_4(i/2)-16E_4(i)+16E_4(2i)-E_2(i/2)^2+16E_2(i)^2-16E_2(2i)^2.
\]
Using \eqref{eq:E4half}, this becomes
\[
144S_1=32E_4(2i)-16E_4(i)-E_2(i/2)^2+16E_2(i)^2-16E_2(2i)^2.
\]
Now substitute \eqref{eq:E2i}, \eqref{eq:E22i}, \eqref{eq:E2half}, \eqref{eq:E42i}, and \eqref{eq:E4i}:
\[
E_2(i)=\frac{3}{\pi},
\qquad
E_2(2i)=\frac{A+B}{2}+\frac{3}{2\pi},
\qquad
E_2(i/2)=-2(A+B)+\frac{6}{\pi},
\]
\[
E_4(2i)=A^2-AB+B^2,
\qquad
E_4(i)=48AB.
\]
A direct expansion gives
\[
144S_1=24A^2-816AB+24B^2+\frac{72}{\pi^2}.
\]
Factor out $24$ and use \eqref{eq:ABrel}:
\[
144S_1=24\bigl(A^2-34AB+B^2\bigr)+\frac{72}{\pi^2}=\frac{72}{\pi^2}.
\]
Therefore
\[
S_1=\frac1{2\pi^2}.
\]

\medskip
\noindent\textbf{Case $m=2$.}
By \eqref{eq:L2},
\[
360S_2=\bigl(E_2E_4-E_6\bigr)(i/2)-44\bigl(E_2E_4-E_6\bigr)(i)+64\bigl(E_2E_4-E_6\bigr)(2i).
\]
Since $E_6(i)=0$, $E_4(i/2)=16E_4(2i)$, and $E_6(i/2)=-64E_6(2i)$, the $E_6$-terms cancel and we obtain
\[
\begin{aligned}
360S_2
&=E_2(i/2)E_4(i/2)-44E_2(i)E_4(i)+64E_2(2i)E_4(2i)\\
&=16E_4(2i)\bigl(E_2(i/2)+4E_2(2i)\bigr)-44E_2(i)E_4(i).
\end{aligned}
\]
By \eqref{eq:E2i}, \eqref{eq:E2half}, and \eqref{eq:E22i},
\[
E_2(i/2)+4E_2(2i)=\frac{12}{\pi}.
\]
Hence
\[
360S_2=\frac{12}{\pi}\bigl(16E_4(2i)-11E_4(i)\bigr).
\]
Using \eqref{eq:E42i} and \eqref{eq:E4i},
\[
16E_4(2i)-11E_4(i)=16(A^2-AB+B^2)-11\cdot 48AB=16(A^2-34AB+B^2)=0
\]
by \eqref{eq:ABrel}. Therefore
\[
S_2=0.
\]

\medskip
\noindent\textbf{Case $m=3$.}
By \eqref{eq:L3},
\[
504S_3=\bigl(E_4^2-E_2E_6\bigr)(i/2)-116\bigl(E_4^2-E_2E_6\bigr)(i)+256\bigl(E_4^2-E_2E_6\bigr)(2i).
\]
Using $E_6(i)=0$, $E_4(i/2)=16E_4(2i)$, and $E_6(i/2)=-64E_6(2i)$, we get
\[
\begin{aligned}
504S_3
&=256E_4(2i)^2-116E_4(i)^2+64E_2(i/2)E_6(2i)+256E_4(2i)^2-256E_2(2i)E_6(2i)\\
&=512E_4(2i)^2-116E_4(i)^2+64\bigl(E_2(i/2)-4E_2(2i)\bigr)E_6(2i).
\end{aligned}
\]
By \eqref{eq:E2half} and \eqref{eq:E22i},
\[
E_2(i/2)-4E_2(2i)=-4(A+B).
\]
Substituting \eqref{eq:E42i}, \eqref{eq:E4i}, and \eqref{eq:E62i}, we obtain
\[
\begin{aligned}
504S_3
&=512(A^2-AB+B^2)^2-116(48AB)^2\\
&\qquad\relax -128(A+B)^2(2A-B)(A-2B).
\end{aligned}
\]
A straightforward expansion and factorization yields
\[
504S_3=128\bigl(A^2-34AB+B^2\bigr)\bigl(2A^2+61AB+2B^2\bigr).
\]
By \eqref{eq:ABrel}, the first factor vanishes. Hence
\[
S_3=0.
\]
This completes the proof.
\end{proof}
The same mechanism explains why these identities are naturally related to modular and quasimodular forms of level $2$: after rewriting
\[
\frac{1}{\cosh x-1}=2\sum_{r=1}^\infty r e^{-rx},
\]
the sums become linear combinations of values of $DE_{2m}$ at the CM points $i/2$, $i$, and $2i$. The cases $m=0,1,2,3$ collapse because of the single algebraic relation \eqref{eq:ABrel} among the theta-constants at $2i$.

\begin{proof}[Proof of Theorem~\ref{thm:squared}]
For $a>0$, define
\[
T(a):=\sum_{n=1}^\infty \frac{1}{(\cosh(a\pi n)-1)^2}.
\]
Then the required identity is
\[
4T(1)-55T(2)+16T(4)=\frac{77-234/\pi}{72}.
\]

We first express $T(a)$ in terms of Eisenstein series. Since
\[
\cosh x-1=\frac{(1-e^{-x})^2}{2e^{-x}},
\]
we have
\[
\frac{1}{(\cosh x-1)^2}
=
\frac{4e^{-2x}}{(1-e^{-x})^4}.
\]
Now
\[
\frac{1}{(1-y)^4}=\sum_{r=0}^\infty \binom{r+3}{3}y^r,
\]
so with $y=e^{-x}$,
\[
\frac{4e^{-2x}}{(1-e^{-x})^4}
=
4\sum_{r=0}^\infty \binom{r+3}{3}e^{-(r+2)x}.
\]
Writing $m=r+2$, we obtain
\[
\frac{1}{(\cosh x-1)^2}
=
4\sum_{m=2}^\infty \binom{m+1}{3}e^{-mx}
=
\frac23\sum_{m=1}^\infty (m^3-m)e^{-mx}.
\]
Therefore, if
\[
q=e^{-a\pi}=e^{2\pi i\tau},
\qquad \tau=\frac{ia}{2},
\]
then
\[
T(a)
=
\frac23\sum_{n=1}^\infty\sum_{m=1}^\infty (m^3-m)q^{mn}.
\]
Collecting the coefficient of $q^N$ gives
\[
\sum_{m\mid N}(m^3-m)=\sigma_3(N)-\sigma_1(N),
\]
hence
\[
T(a)=\frac23\sum_{N=1}^\infty \bigl(\sigma_3(N)-\sigma_1(N)\bigr)q^N.
\]
Using
\[
E_2(\tau)=1-24\sum_{N=1}^\infty \sigma_1(N)q^N,
\qquad
E_4(\tau)=1+240\sum_{N=1}^\infty \sigma_3(N)q^N,
\]
we get
\[
\sum_{N=1}^\infty \sigma_1(N)q^N=\frac{1-E_2(\tau)}{24},
\qquad
\sum_{N=1}^\infty \sigma_3(N)q^N=\frac{E_4(\tau)-1}{240}.
\]
Therefore
\[
T(a)
=
\frac23\left(\frac{E_4(\tau)-1}{240}-\frac{1-E_2(\tau)}{24}\right)
=
\frac{E_4(\tau)+10E_2(\tau)-11}{360}.
\]

Now define
\[
C:=4T(1)-55T(2)+16T(4).
\]
Since $a=1,2,4$ correspond to $\tau=i/2,i,2i$, we obtain
\[
360C
=
4\bigl(E_4(i/2)+10E_2(i/2)-11\bigr)
-55\bigl(E_4(i)+10E_2(i)-11\bigr)
+16\bigl(E_4(2i)+10E_2(2i)-11\bigr).
\]
Expanding,
\[
360C
=
\bigl(4E_4(i/2)-55E_4(i)+16E_4(2i)\bigr)
+
\bigl(40E_2(i/2)-550E_2(i)+160E_2(2i)\bigr)
+385.
\]

We now simplify the $E_2$-part. By \eqref{eq:E2i}, \eqref{eq:E22i}, and \eqref{eq:E2half},
\[
E_2(i)=\frac{3}{\pi},
\qquad
E_2(i/2)+4E_2(2i)=\frac{12}{\pi}.
\]
Hence
\[
40E_2(i/2)-550E_2(i)+160E_2(2i)
=
40\bigl(E_2(i/2)+4E_2(2i)\bigr)-550E_2(i)
=
40\cdot \frac{12}{\pi}-550\cdot \frac{3}{\pi}
=
-\frac{1170}{\pi}.
\]

Next we simplify the $E_4$-part. By \eqref{eq:E4half},
\[
E_4(i/2)=16E_4(2i),
\]
so
\[
4E_4(i/2)-55E_4(i)+16E_4(2i)
=
80E_4(2i)-55E_4(i)
=
5\bigl(16E_4(2i)-11E_4(i)\bigr).
\]
Using \eqref{eq:E42i}, \eqref{eq:E4i}, and \eqref{eq:ABrel}, we obtain
\[
16E_4(2i)-11E_4(i)
=
16(A^2-AB+B^2)-11\cdot 48AB
=
16(A^2-34AB+B^2)
=
0.
\]
Therefore the entire $E_4$-part vanishes.

Substituting back, we find
\[
360C=385-\frac{1170}{\pi}.
\]
Thus
\[
C=\frac{385-1170/\pi}{360}
=\frac{77-234/\pi}{72}.
\]
Since $C=4T(1)-55T(2)+16T(4)$, this is exactly
\[
\sum_{n=1}^\infty \left(
\frac{4}{(\cosh(n\pi)-1)^2}
-\frac{55}{(\cosh(2n\pi)-1)^2}
+\frac{16}{(\cosh(4n\pi)-1)^2}
\right)
=
\frac{77-234/\pi}{72}.
\]
This completes the proof.
\end{proof}

\subsection{An alternative proof of the cases \(m=0,1,2,3\)}\label{sec:alt}
Using
\[
\cosh(2u)-1=2\sinh^2(u),
\]
the identity
\begin{align*}
S_0&=\sum_{n=1}^{\infty}\left(
\frac{1}{\cosh(n\pi)-1}
-\frac{4}{\cosh(2n\pi)-1}
+\frac{4}{\cosh(4n\pi)-1}
\right)=\frac{1}{12},\\
S_1&=\sum_{n=1}^{\infty}\left(
\frac{n^2}{\cosh(n\pi)-1}
-\frac{16n^2}{\cosh(2n\pi)-1}
+\frac{16n^2}{\cosh(4n\pi)-1}
\right)=\frac{1}{2\pi^2},\\
S_2&=\sum_{n=1}^\infty\left(\frac{n^4}{\cosh(n\pi)-1}-\frac{44n^4}{\cosh(2n\pi)-1}+\frac{64n^4}{\cosh(4n\pi)-1}\right)=0,\\
S_3&=\sum_{n=1}^\infty\left(\frac{n^6}{\cosh(n\pi)-1}-\frac{116n^6}{\cosh(2n\pi)-1}+\frac{256n^6}{\cosh(4n\pi)-1}\right)=0
\end{align*}
are equivalent to
\begin{align*}
&\frac1{2}S_{0,2}\!\left(\frac{\pi}{2}\right)-2S_{0,2}(\pi)+2S_{0,2}(2\pi)=\frac1{12},\\
&\frac12\,S_{2,2}\!\left(\frac{\pi}{2}\right)-8S_{2,2}(\pi)+8S_{2,2}(2\pi)=\frac{1}{2\pi^2},\\
&\frac1{2}S_{4,2}\!\left(\frac{\pi}{2}\right)-22S_{4,2}(\pi)+32S_{4,2}(2\pi)=0,\\
&\frac1{2}S_{6,2}\!\left(\frac{\pi}{2}\right)-58S_{6,2}(\pi)+128S_{6,2}(2\pi)=0,
\end{align*}
where
\[
q=e^{-y},
\qquad
S_{2p,2}(y):=4\sum_{n=1}^{\infty}\frac{n^{2p}q^{2n}}{(1-q^{2n})^2}
=\sum_{n=1}^{\infty}\frac{n^{2p}}{\sinh^2(ny)}.
\]

First, we prove that the first formula for \( S_0 \) holds. From \cite[eq. (101)]{XuZhao-2022}, we have
\begin{align*}
aS_{0,2}(a)+bS_{0,2}(b)-\frac1{6}(a+b)+1=0\quad (ab=\pi^2).
\end{align*}
Let \((a, b) = (\pi/2, 2\pi)\) and \((\pi, \pi)\). Then we have
\begin{align*}
\frac1{2}S_{0,2}\!\left(\frac{\pi}{2}\right)+2S_{0,2}(2\pi)=\frac5{12}-\frac1{\pi}
\end{align*}
and
\begin{align*}
S_{0,2}(\pi)=\frac1{6}-\frac1{2\pi}.
\end{align*}
Hence, one obtains 
\begin{align*}
\frac1{2}S_{0,2}\!\left(\frac{\pi}{2}\right)-2S_{0,2}(\pi)+2S_{0,2}(2\pi)=\frac1{12}.
\end{align*}

Now introduce Ramanujan's parameters (see \cite{Berndt2016})
\[
z=z(x):=\Hyp\!\left(\frac12,\frac12;1;x\right),
\qquad
y=y(x):=\pi\,
\frac{\Hyp\!\left(\frac12,\frac12;1;1-x\right)}
     {\Hyp\!\left(\frac12,\frac12;1;x\right)},
\qquad
z'=\frac{dz}{dx}.
\]
Then from \cite{XuZhao-2022}, we have 
\begin{align*}
&S_{2,2}\!\left(\frac{y}{2}\right)=\frac{x(1-x)z^2}{3}\Bigl(z^2-(1-5x)zz'-6x(1-x)(z')^2\Bigr),\\
&S_{4,2}\!\left(\frac{y}{2}\right)=\frac{x(1-x)z^5}{30}\Bigl(2(1+14x+x^2)z'+(7+x)z\Bigr),\\
&S_{6,2}\!\left(\frac{y}{2}\right)=\frac{x(1-x)z^7}{42}\Bigl(-2(1-33x-33x^2+x^3)z'+(11+22x-x^2)z\Bigr).
\end{align*}

To obtain the formulas for \(S_{m,2}(y)\) \ and\ \(S_{m,2}(2y)\)\ $(m=2,4,6)$, apply the quadratic transformation from \cite[Theorem 1]{XuZhao-2022}:
if
\[
r=\sqrt{1-x},
\qquad
x_1=\left(\frac{1-r}{1+r}\right)^2,
\qquad
z_1=\frac12 z(1+r),
\]
then
\[
z_1'=
\frac{(1+r)^3}{4(1-r)}
\left(
-\frac z2+(r+1-x)z'
\right),
\]
and \(y(x_1)=2y(x)\). Therefore

\begin{align*}
&S_{2,2}(y)=\frac{x(1-x)z^2}{24}\Bigl(z^2-4(1-2x)zz'-12x(1-x)(z')^2\Bigr),\\
&S_{2,2}(2y)=\frac{x(1-x)z^2}{48}\Bigl(z^2+(5x-4)zz'-6x(1-x)(z')^2\Bigr),\\
&S_{4,2}(y)=\frac{(1-x)xz^5}{120}\Bigl(4(1-x+x^2)z'+(2x-1)z\Bigr),\\
&S_{4,2}(2y)=\frac{(1-x)xz^5}{1920}\Bigl(2(16-16x+x^2)z'+(x-8)z\Bigr),\\
&S_{6,2}(y)=\frac{xz^7}{168}\Bigl((1+x-4x^2+2x^3)z+2(-2+5x-5x^3+2x^4)z'\Bigr),\\
&S_{6,2}(2y)=\frac{(x-1)xz^7}{10752}\Bigl((-32+20x+x^2)z+2(64-96x+30x^2+x^3)z'\Bigr).
\end{align*}
For explicit evaluations of certain hyperbolic series with denominators raised to the first power, the reader is referred to \cite{B1991} and \cite{Rama1916}.

Now set \(x=\frac12\). Then \(y=\pi\), and
\[
z\!\left(\frac12\right)=\frac{\Ga\!\left(\frac14\right)^2}{2\pi^{3/2}},
\qquad
z'\!\left(\frac12\right)=\frac{4\sqrt{\pi}}{\Ga\!\left(\frac14\right)^2}.
\]
Substituting into the three formulas above yields
\begin{align*}
&S_{2,2}\!\left(\frac{\pi}{2}\right)=-\frac{1}{2\pi^2}+\frac{\Ga\!\left(\frac14\right)^4}{16\pi^4}+\frac{\Ga\!\left(\frac14\right)^8}{192\pi^6},\\
&S_{2,2}(\pi)=-\frac{1}{8\pi^2}+\frac{\Ga\!\left(\frac14\right)^8}{1536\pi^6},\\
&S_{2,2}(2\pi)=-\frac{1}{32\pi^2}-\frac{\Ga\!\left(\frac14\right)^4}{256\pi^4}+\frac{\Ga\!\left(\frac14\right)^8}
{3072\pi^6},\\
&S_{4,2}\!\left(\frac{\pi}{2}\right)=\frac{11\Ga\!\left(\frac14\right)^8}{640\pi^7}+\frac{\Ga\!\left(\frac14\right)^{12}}{1024\pi^9},\\
&S_{4,2}(\pi)=\frac{\Ga\!\left(\frac14\right)^8}{1280\pi^7},\\
&S_{4,2}(2\pi)=\frac{11\Ga\!\left(\frac14\right)^8}{40960\pi^7}-\frac{\Ga\!\left(\frac14\right)^{12}}{65536\pi^9},\\
&S_{6,2}\!\left(\frac{\pi}{2}\right)=\frac{9\Ga\!\left(\frac14\right)^{12}}{1024\pi^{10}}+\frac{29\Ga\!\left(\frac14\right)^{16}}{57344\pi^{12}},\\
&S_{6,2}(\pi)=\frac{\Ga\!\left(\frac14\right)^{16}}{114688\pi^{12}},\\
&S_{6,2}(2\pi)=-\frac{9\Ga\!\left(\frac14\right)^{12}}{262144\pi^{10}}+\frac{29\Ga\!\left(\frac14\right)^{16}}{14680064\pi^{12}}.
\end{align*}

Hence
\begin{align*}
\frac12\,S_{2,2}\!\left(\frac{\pi}{2}\right)-8S_{2,2}(\pi)+8S_{2,2}(2\pi)
&=
\frac{\Ga\!\left(\frac14\right)^8+12\Ga\!\left(\frac14\right)^4\pi^2-96\pi^4}{384\pi^6}
-\frac{\Ga\!\left(\frac14\right)^8-192\pi^4}{192\pi^6} \\
&\qquad
+\frac{\Ga\!\left(\frac14\right)^8-12\Ga\!\left(\frac14\right)^4\pi^2-96\pi^4}{384\pi^6} \\
&=
\frac{\Ga\!\left(\frac14\right)^8-96\pi^4}{192\pi^6}
-\frac{\Ga\!\left(\frac14\right)^8-192\pi^4}{192\pi^6} \\
&=
\frac{96\pi^4}{192\pi^6}
=
\frac{1}{2\pi^2},
\end{align*}
and 
\begin{align*}
&\frac1{2}S_{4,2}\!\left(\frac{\pi}{2}\right)-22S_{4,2}(\pi)+32S_{4,2}(2\pi)=0,\\
&\frac1{2}S_{6,2}\!\left(\frac{\pi}{2}\right)-58S_{6,2}(\pi)+128S_{6,2}(2\pi)=0.
\end{align*}

Therefore
\begin{align*}
S_1&=\sum_{n=1}^{\infty}\left(
\frac{n^2}{\cosh(n\pi)-1}
-\frac{16n^2}{\cosh(2n\pi)-1}
+\frac{16n^2}{\cosh(4n\pi)-1}
\right)=\frac{1}{2\pi^2},\\
S_2&=\sum_{n=1}^\infty\left(\frac{n^4}{\cosh(n\pi)-1}-\frac{44n^4}{\cosh(2n\pi)-1}+\frac{64n^4}{\cosh(4n\pi)-1}\right)=0,\\
S_3&=\sum_{n=1}^\infty\left(\frac{n^6}{\cosh(n\pi)-1}-\frac{116n^6}{\cosh(2n\pi)-1}+\frac{256n^6}{\cosh(4n\pi)-1}\right)=0
\end{align*}

\section{Proof of Theorem~\ref{thm:main} in full generality}\label{sec:last}

We now prove Theorem~\ref{thm:main} for all \(m\ge 0\). Recall that
\[
S_m=\sum_{n=1}^\infty\left(
\frac{n^{2m}}{\cosh(n\pi)-1}
-\frac{c_m n^{2m}}{\cosh(2n\pi)-1}
+\frac{2^{2m+2}n^{2m}}{\cosh(4n\pi)-1}
\right),
\]
where
\[
c_m=2^{2m+1}-(-1)^{m(m+1)/2}2^{m+1}+4.
\]
We need to prove that for all \(m\ge 0\), one has
\[
S_0=\frac1{12},\qquad S_1=\frac1{2\pi^2},\qquad S_m=0\ \text{for all }m>1.
\]

\begin{proof}
The cases \(m=0\) and \(m=1\) were proved above, so it remains to show that \(S_m=0\) for every \(m>1\).

Let
\[
k:=2m,\qquad \varepsilon_m:=(-1)^{m(m+1)/2}.
\]
Then \(k\ge 4\). By Lemma~\ref{lem:Lm}, for \(\tau=ia/2\),
\[
L_m(a)= -\frac{B_k}{k}\,D E_k(\tau),
\qquad D=\frac{1}{2\pi i}\frac{d}{d\tau}.
\]
Hence
\[
S_m=0
\]
is equivalent to
\[
D E_k(i/2)-c_m D E_k(i)+2^{k+2}D E_k(2i)=0.
\]
Since \(D=\frac{1}{2\pi i}\frac{d}{d\tau}\), this is equivalent to
\[
E_k'(i/2)-c_m E_k'(i)+2^{k+2}E_k'(2i)=0.
\]
It is convenient to replace \(E_k\) by the classical Eisenstein series
\[
G_k(\tau):=2\zeta(k)E_k(\tau)
=\sum_{(u,v)\in\ZZ^2\setminus\{(0,0)\}}(u+v\tau)^{-k}.
\]
Since \(2\zeta(k)\neq 0\), it is enough to prove
\begin{equation}
G_k'(i/2)-c_m G_k'(i)+2^{k+2}G_k'(2i)=0.
\label{eq:target-general}
\end{equation}

For \(k\ge 4\), the defining series for \(G_k\) and \(G_k'\) converge absolutely and locally uniformly on \(\mathfrak H\), so all rearrangements and termwise differentiations below are justified.

\medskip
\noindent\textbf{Step 1: parity class sums.}
For \(\alpha,\beta\in\{0,1\}\), define
\[
\Sigma_{\alpha,\beta}(\tau):=
\sum_{r,s\in\ZZ}^{*_{\alpha,\beta}}
\bigl(2r+\alpha +(2s+\beta)\tau\bigr)^{-k},
\]
where the only omitted term is \(r=s=0\) when \((\alpha,\beta)=(0,0)\).
Then
\[
G_k(\tau)=\Sigma_{00}(\tau)+\Sigma_{10}(\tau)+\Sigma_{01}(\tau)+\Sigma_{11}(\tau),
\]
because every lattice point \((u,v)\in\ZZ^2\) falls into exactly one parity class modulo \(2\). Moreover,
\[
G_k(2\tau)=\Sigma_{00}(\tau)+\Sigma_{10}(\tau),
\]
because \(G_k(2\tau)\) is the subsum over lattice points whose coefficient of \(\tau\) is even, that is, over the classes with \(\beta=0\).

We abbreviate
\[
\Sigma_{\alpha,\beta}:=\Sigma_{\alpha,\beta}(i),
\qquad
\Sigma_{\alpha,\beta}':=\frac{d}{d\tau}\Sigma_{\alpha,\beta}(\tau)\Big|_{\tau=i}.
\]

The class \((0,0)\) is just a dilation, so
\begin{equation}
\Sigma_{00}(\tau)=2^{-k}G_k(\tau),
\qquad
\Sigma_{00}'=2^{-k}G_k'(i).
\label{eq:s00}
\end{equation}
The classes with equal parity satisfy
\[
\Sigma_{00}+\Sigma_{11}
=
\sum_{\substack{u,v\in\ZZ\\u\equiv v\!\!\!\pmod2}}^{\!\!\!\!\!\!\!\!\prime}
(u+vi)^{-k}.
\]
The change of variables
\[
u=a-b,\qquad v=a+b
\]
is a bijection from \(\ZZ^2\) onto the set of pairs \((u,v)\in\ZZ^2\) with \(u\equiv v\pmod 2\). Under this change of variables,
\[
u+vi=(a-b)+(a+b)i=(1+i)(a+bi),
\]
hence
\begin{equation}
\Sigma_{00}+\Sigma_{11}=(1+i)^{-k}G_k(i).
\label{eq:same-parity}
\end{equation}

Also, multiplication by \(i\) sends
\[
2r+1+2si \longmapsto i(2r+1+2si)=-2s+(2r+1)i,
\]
which is exactly the class \((0,1)\). Therefore
\begin{equation}
\Sigma_{01}=i^{-k}\Sigma_{10}=(-1)^m\Sigma_{10}.
\label{eq:s01-s10}
\end{equation}

\medskip
\noindent\textbf{Step 2: the case \(m\) even.}
Assume first that \(m\) is even. Then \(k\equiv 0\pmod4\), so \((-1)^m=1\), and \eqref{eq:s01-s10} gives
\[
\Sigma_{01}=\Sigma_{10}.
\]
Therefore
\[
G_k(i)=\Sigma_{00}+2\Sigma_{10}+\Sigma_{11}.
\]
Since \(G_k(2i)=\Sigma_{00}+\Sigma_{10}\), we obtain
\[
G_k(i)=2G_k(2i)-\Sigma_{00}+\Sigma_{11}.
\]
Using \eqref{eq:s00} and \eqref{eq:same-parity},
\[
\Sigma_{11}=(1+i)^{-k}G_k(i)-2^{-k}G_k(i).
\]
Hence
\[
G_k(i)=2G_k(2i)+\bigl((1+i)^{-k}-2^{1-k}\bigr)G_k(i).
\]
Now
\[
(1+i)^{-k}=2^{-m}(-i)^m.
\]
Since \(m\) is even, \((-i)^m=\varepsilon_m\). Thus
\[
\bigl(1-\varepsilon_m 2^{-m}+2^{1-k}\bigr)G_k(i)=2G_k(2i).
\]
Multiplying by \(2^{k+1}\), and recalling that \(k=2m\), yields
\begin{equation}
c_m\,G_k(i)=2^{k+2}G_k(2i).
\label{eq:even-value}
\end{equation}

We now turn to derivatives. The modular transformation
\[
G_k\!\left(-\frac1\tau\right)=\tau^k G_k(\tau)
\]
implies, by differentiation,
\begin{equation}
G_k'\!\left(-\frac1\tau\right)=\tau^{k+2}G_k'(\tau)+k\tau^{k+1}G_k(\tau).
\label{eq:deriv-transform}
\end{equation}
Evaluating at \(\tau=i\), and using \(i^k=1\) because \(k\equiv 0\pmod4\), we get
\[
G_k'(i)=\frac{k i}{2}G_k(i).
\]
Evaluating \eqref{eq:deriv-transform} at \(\tau=2i\), and using \(i/2=-1/(2i)\), gives
\[
G_k'(i/2)=(2i)^{k+2}G_k'(2i)+k(2i)^{k+1}G_k(2i).
\]
Since \(k\equiv 0\pmod4\),
\[
(2i)^{k+2}=-2^{k+2},
\qquad
(2i)^{k+1}=2^{k+1}i.
\]
Therefore
\[
G_k'(i/2)+2^{k+2}G_k'(2i)=k2^{k+1}i\,G_k(2i).
\]
Subtracting \(c_m G_k'(i)=\frac{k i}{2}c_m G_k(i)\), we obtain
\[
G_k'(i/2)-c_m G_k'(i)+2^{k+2}G_k'(2i)
=
k i\left(2^{k+1}G_k(2i)-\frac{c_m}{2}G_k(i)\right).
\]
By \eqref{eq:even-value}, the bracket vanishes. Thus \eqref{eq:target-general} holds when \(m\) is even.

\medskip
\noindent\textbf{Step 3: the case \(m\) odd.}
Now assume that \(m\) is odd. Then \(k\equiv 2\pmod4\). From modularity at the fixed point \(i\),
\[
G_k(i)=i^k G_k(i)=-G_k(i),
\]
hence
\begin{equation}
G_k(i)=0.
\label{eq:Gk-i-zero}
\end{equation}
By \eqref{eq:s00}, \eqref{eq:same-parity}, and \eqref{eq:s01-s10}, this implies
\[
\Sigma_{00}=0,\qquad \Sigma_{11}=0,\qquad \Sigma_{01}=-\Sigma_{10}.
\]
Since \(G_k(2i)=\Sigma_{00}+\Sigma_{10}\), we obtain
\begin{equation}
G_k(2i)=\Sigma_{10}.
\label{eq:G2i-S10}
\end{equation}

Next we compute the derivatives of the class sums.

First, differentiating the identity \(G_k(2\tau)=\Sigma_{00}(\tau)+\Sigma_{10}(\tau)\) at \(\tau=i\), we get
\begin{equation}
2G_k'(2i)=\Sigma_{00}'+\Sigma_{10}'.
\label{eq:s10prime}
\end{equation}
Using \eqref{eq:s00}, this gives
\begin{equation}
\Sigma_{10}'=2G_k'(2i)-2^{-k}G_k'(i).
\label{eq:s10prime-explicit}
\end{equation}

Second, consider
\[
\Sigma_{00}'+\Sigma_{11}'
=
-k\sum_{\substack{u,v\in\ZZ\\u\equiv v\!\!\!\pmod2}}^{\!\!\!\!\!\!\!\!\prime}
v\,(u+vi)^{-k-1}.
\]
Again write \(u=a-b\), \(v=a+b\), so \(u+vi=(1+i)(a+bi)\). Then
\[
\Sigma_{00}'+\Sigma_{11}'
=
-k(1+i)^{-k-1}
\sum_{(a,b)\ne(0,0)}
(a+b)(a+bi)^{-k-1}.
\]
Now
\[
a+b=\frac{1-i}{2}(a+bi)+\frac{1+i}{2}(a-bi),
\]
hence
\[
\Sigma_{00}'+\Sigma_{11}'
=
-\frac{k(1+i)^{-k-1}}{2}
\sum_{(a,b)\ne(0,0)}
\left((1-i)(a+bi)^{-k}+(1+i)(a-bi)(a+bi)^{-k-1}\right).
\]
Because \(G_k(i)=0\), the first sum vanishes, and therefore
\[
\Sigma_{00}'+\Sigma_{11}'
=
-\frac{k(1+i)^{-k}}{2}
\sum_{(a,b)\ne(0,0)}
(a-bi)(a+bi)^{-k-1}.
\]
Set
\[
A:=\sum_{(a,b)\ne(0,0)} a(a+bi)^{-k-1},
\qquad
B:=\sum_{(a,b)\ne(0,0)} b(a+bi)^{-k-1}.
\]
Then the last sum is \(A-iB\). Under the rotation \((a,b)\mapsto(-b,a)\), we have
\[
A
=
\sum_{(a,b)\ne(0,0)} (-b)(-b+ai)^{-k-1}
=
i^{-k-1}(-1)\sum_{(a,b)\ne(0,0)} b(a+bi)^{-k-1}.
\]
Since \(k\equiv 2\pmod4\), one has \(i^{-k-1}=i\), so
\[
A=-iB.
\]
Thus
\[
A-iB=-2iB.
\]
On the other hand,
\[
G_k'(i)=-k\sum_{(a,b)\ne(0,0)} b(a+bi)^{-k-1}=-kB,
\]
so \(B=-G_k'(i)/k\). Therefore
\[
\Sigma_{00}'+\Sigma_{11}'
=
-i(1+i)^{-k}G_k'(i).
\]
Since
\[
(1+i)^{-k}=2^{-m}(-i)^m,
\]
and \(m\) is odd, one checks that
\[
-i(1+i)^{-k}=\varepsilon_m 2^{-m}.
\]
Hence
\begin{equation}
\Sigma_{00}'+\Sigma_{11}'=\varepsilon_m 2^{-m}G_k'(i).
\label{eq:s00plus11prime}
\end{equation}
Using \eqref{eq:s00}, this yields
\begin{equation}
\Sigma_{11}'=\bigl(\varepsilon_m 2^{-m}-2^{-k}\bigr)G_k'(i).
\label{eq:s11prime}
\end{equation}

Third, we relate \(\Sigma_{01}'\) and \(\Sigma_{10}'\). Every element of the class \((0,1)\) is \(i z\) with \(z\) in the class \((1,0)\). Hence
\[
\Sigma_{01}'
=
-k\sum_{r,s\in\ZZ}(2r+1)\bigl(i(2r+1+2si)\bigr)^{-k-1}.
\]
Since \(k\equiv 2\pmod4\), we have \(i^{-k-1}=i\). Writing
\[
z:=2r+1+2si,
\qquad
2r+1=z-2si,
\]
we obtain
\[
\Sigma_{01}'
=
-k i\sum z^{-k}
-2k\sum s\,z^{-k-1}.
\]
The first sum is \(\Sigma_{10}\), and the second is exactly \(\Sigma_{10}'\). Therefore
\begin{equation}
\Sigma_{01}'=\Sigma_{10}'-k i\,\Sigma_{10}.
\label{eq:s01prime}
\end{equation}
Using \eqref{eq:G2i-S10}, this becomes
\[
\Sigma_{01}'=\Sigma_{10}'-k i\,G_k(2i).
\]

Now combine everything:
\[
G_k'(i)=\Sigma_{00}'+\Sigma_{10}'+\Sigma_{01}'+\Sigma_{11}'.
\]
Substituting \eqref{eq:s10prime-explicit}, \eqref{eq:s11prime}, and \eqref{eq:s01prime}, we find
\[
\begin{aligned}
G_k'(i)
&=
2^{-k}G_k'(i)
+\Sigma_{10}'
+\bigl(\Sigma_{10}'-k i\,G_k(2i)\bigr)
+\bigl(\varepsilon_m 2^{-m}-2^{-k}\bigr)G_k'(i)\\
&=
2\Sigma_{10}'-k i\,G_k(2i)+\varepsilon_m 2^{-m}G_k'(i).
\end{aligned}
\]
Using \eqref{eq:s10prime-explicit},
\[
G_k'(i)
=
4G_k'(2i)-2^{1-k}G_k'(i)-k i\,G_k(2i)+\varepsilon_m 2^{-m}G_k'(i).
\]
Hence
\[
\bigl(1-\varepsilon_m 2^{-m}+2^{1-k}\bigr)G_k'(i)
=
4G_k'(2i)-k i\,G_k(2i).
\]
Multiplying by \(2^{k+1}\) gives
\begin{equation}
c_m G_k'(i)=2^{k+3}G_k'(2i)-k2^{k+1}i\,G_k(2i).
\label{eq:odd-deriv-identity}
\end{equation}

Finally, apply \eqref{eq:deriv-transform} at \(\tau=2i\):
\[
G_k'(i/2)=(2i)^{k+2}G_k'(2i)+k(2i)^{k+1}G_k(2i).
\]
Since \(k\equiv 2\pmod4\),
\[
(2i)^{k+2}=2^{k+2},
\qquad
(2i)^{k+1}=-2^{k+1}i.
\]
Therefore
\[
G_k'(i/2)=2^{k+2}G_k'(2i)-k2^{k+1}i\,G_k(2i).
\]
Subtracting \(c_m G_k'(i)\) and adding \(2^{k+2}G_k'(2i)\), and then using \eqref{eq:odd-deriv-identity}, we get
\[
G_k'(i/2)-c_m G_k'(i)+2^{k+2}G_k'(2i)=0.
\]
Thus \eqref{eq:target-general} holds when \(m\) is odd as well.

We have therefore proved \eqref{eq:target-general} for every \(m>1\), hence \(S_m=0\) for all \(m>1\). Together with the already established cases \(m=0\) and \(m=1\), this proves the theorem.
\end{proof}










\end{document}